\documentclass[twoside,a4paper,UTF8]{ctexart}   
\usepackage{amsmath,amsthm,makeidx,amssymb,tikz-cd}     
\usepackage{amsfonts}                     
\usepackage{graphicx}
\usepackage{xcolor}
\usepackage{lineno}
\usepackage{hyperref}
\usepackage{mathrsfs}
\usepackage{enumerate}
\usepackage[all]{xy}
\input scrload.tex    
\def\sec#1{
\vskip.12in \noindent{\Large\bf\zihao{-4}\heiti #1} \vskip.1in}

\catcode`@=11 \long\def\@makefntext#1{\parindent 1em\noindent
\hbox to 0pt{\hss$^{}$}#1} \catcode`\@=12

\newfont{\htxt}{eufm10 scaled \magstep0}
   
scaled \magstep0

\def\text#1{\hbox{\,#1\,}}

\def\text#1{\hbox{\,#1\,}}

\usepackage{stmaryrd}

\theoremstyle{definition}
\newtheorem{defn}{Definition}[section]
\newtheorem{lemma}[defn]{Lemma}
\newtheorem{theorem}[defn]{Theorem}
\newtheorem{proposition}[defn]{Proposition}

\newtheorem{remark}[defn]{Remark}

\begin{document}
 \abovedisplayskip=10.0pt plus 2.0pt minus 2.0pt
\belowdisplayskip=10.0pt plus 2.0pt minus 2.0pt
\ziju{0.025}

\vspace*{.25in}

\centerline{\LARGE\heiti\bf\zihao{2} One Decomposition of $K_2$-Group for}
\vskip .1in
\centerline{\LARGE\heiti\bf\zihao{2} Certain Quotients over $\mathbb{Z}[G]$ with } \vskip.1in
\centerline{\LARGE\heiti\bf\zihao{2} $G$ a Finite Abelian $p$-Group} \vskip.1in

\centerline{\large\fangsong\zihao{4} ZHANG Yakun$^{1,*}$}\vskip.13in

\baselineskip 12pt{\small {\it(1. School of Mathematics,  Nanjing Audit University, Nanjing 211815,  P. R. China)}}\vskip.23in

\footnotetext{E-mail: $*$ zhangyakun@nau.edu.cn}

{\small \narrower {\bf Abstract:}\ \ This paper investigates the structure of $K_2$-groups for certain quotient rings of the integral group ring $\mathbb{Z}[G]$. Let $G$ be a finite abelian $p$-group with $p$-rank $r$, let $\Gamma$ be the maximal $\mathbb{Z}$-order of $\mathbb{Q}[G]$, and let $\widetilde{G}$ denote the sum of all elements of $G$ in the group ring. By employing the framework of K\"{a}hler differentials, we first determine that the relative $K$-group $K_2(\mathbb{F}_p[G], (\widetilde{G}))$ is an elementary abelian $p$-group of rank $r$ when $|G|>2$. Building upon this result, we establish an explicit isomorphism for $r > 1$:
$$
K_2(\mathbb{Z}[G]/(|G|\Gamma \cap p\mathbb{Z}[G])) \cong K_2(\mathbb{Z}[G]/|G|\Gamma) \oplus K_2(\mathbb{F}_p[G], (\widetilde{G})).
$$

{\bf Key words:}\ \ $K_2$-groups; group rings; Dennis-Stein symbols; K\"{a}hler differentials

 {{\bf\heiti MR(2010) Subject Classification:}}\ \  16S34; 19C20; 19C99 / {{\bf\heiti
 CLC number:}}\ \ O154.3

 {\zihao{-5}{\bf Document code:}}\ \ A \qquad{{\bf\zihao{6} Article ID:}}\ \

} \par

\vskip.15in
\normalsize

\baselineskip 15pt

\sec{0 Introduction}
This paper is a continuation of our previous work \cite{smallK2}. Let $G$ be a finite abelian $p$-group with $p$-rank $r$, and let $\Gamma$ denote the maximal $\mathbb{Z}$-order in the group algebra $\mathbb{Q}[G]$. Set $J = |G|\Gamma$. By Proposition 2.2 in \cite{alperin1985sk}, $\Gamma$ is isomorphic to a direct product of rings of algebraic integers, with the inclusions $J \subseteq \mathbb{Z}[G] \subseteq \Gamma$ holding.

In \cite{ChenHong2014}, the group $K_2(\mathbb{Z}[G]/J)$ was investigated to establish a lower bound for the order of $K_2(\mathbb{Z}[G])$. It was shown that $K_2(\mathbb{Z}[G]/J)$ is trivial if and only if the order of $G$ is square-free. However, determining the explicit structure of $K_2(\mathbb{Z}[G]/J)$ remains an open problem, largely due to the complexity of $\Gamma$ (see \cite[Lemma 3.3]{ChenHong2014}). In our previous work \cite{smallK2}, we provided computational evidence and nontrivial examples for groups $G$ of small order.

Let $I = J \cap p\mathbb{Z}[G]$ and $\widetilde{G} = \sum_{g \in G} g$. These definitions give rise to the following Cartesian square:
\begin{equation} \label{cartesian}
\begin{tikzcd}
\mathbb{Z}[G]/I \arrow[r] \arrow[d] & \mathbb{Z}[G]/J \arrow[d] \\
\mathbb{F}_p[G] \arrow[r]           & \mathbb{F}_p[G]/(\widetilde{G})
\end{tikzcd}
\end{equation}
In the present paper, we establish an explicit relation between the $K_2$-groups of the rings lying in \eqref{cartesian}. The rest of the article is organized as follows: Section 1 is devoted to the necessary preliminaries.

In Section 2, we employ the framework of K\"{a}hler differentials to prove two key results for $|G|>2$, which are then applied to determine the explicit structure of $K_2(\mathbb{F}_p[G],(\widetilde{G}))$ (see Lemma \ref{K2FpG}). Furthermore, Theorem \ref{main-thm1} is used to establish  the excision property for $K_2$​ with respect to the cartesian square \eqref{cartesian}, as described in Proposition \ref{excision}.

As a consequence, in Section 3, we show that the following exact sequences of finite abelian $p$-groups split when $r>1$:
$$
1 \rightarrow  K_2(\mathbb{Z}[G]/I, J/I) \rightarrow K_2(\mathbb{Z}[G]/I) \rightarrow K_2(\mathbb{Z}[G]/J) \rightarrow 1,
$$
$$
1 \rightarrow K_2(\mathbb{F}_p[G], (\widetilde{G})) \rightarrow K_2(\mathbb{F}_p[G]) \rightarrow K_2(\mathbb{F}_p[G]/(\widetilde{G})) \rightarrow 1.
$$
Combining this with the aforementioned excision isomorphism, we establish the decomposition (see Theorem \ref{main-thm2}):
\begin{equation*}\label{eq:main_iso}
K_2(\mathbb{Z}[G]/I) \cong K_2(\mathbb{Z}[G]/J) \oplus \mathrm{Ker}\left(K_2(\mathbb{F}_p[G]) \rightarrow K_2(\mathbb{F}_p[G]/(\widetilde{G}))\right).
\end{equation*}
This result, obtained by extending the methods of \cite{ChenHong2014} to a more general setting, provides a structural decomposition for the $K_2$-groups of the quotient rings of $\mathbb{Z}[G]$ in square \eqref{cartesian}, and extends our previous computational studies to a general theoretical framework.

\sec{1 Some preliminaries}
\stepcounter{section}
According to \cite[Chapter III, \S 5]{weibel2013kbook}, if $I$ is a radical ideal of a commutative ring $R$, then $1-ab$ is a unit for all $a \in R, b \in I$, and the relative $K_2$-group $K_2(R, I)$ is generated by the \textit{Dennis--Stein symbols} $\langle a, b \rangle$ for $(a, b) \in (R \times I) \cup (I \times R)$. These symbols satisfy the following defining relations (written multiplicatively):
\begin{enumerate}
    \item[(DS1)] $\langle a, b \rangle \langle b, a \rangle = 1$
    \item[(DS2)] $\langle a, b \rangle \langle a, c \rangle = \langle a, b+c-abc \rangle$
    \item[(DS3)] $\langle a, bc \rangle = \langle ab, c \rangle \langle ac, b \rangle$ \quad (where at least one of $a, b, c$ belongs to $I$)
\end{enumerate}	

\begin{remark} \label{remak1}
In the present study, the ideal $I$ satisfies $I^2 = 0$, which naturally ensures the radical condition. Under this setting, the symbol $\langle a, b \rangle$ used here (following \cite{weibel2013kbook}) coincides with the notation $\langle -a, b \rangle^{-1}$ in earlier literature such as \cite{weibel1980k2}. Furthermore, the term $abc$ in (DS2) vanishes, which reduces the relation to $\langle a, b+c \rangle = \langle a, b \rangle \langle a, c \rangle$. By the skew-symmetry from (DS1), the symbol $\langle a, b \rangle$ is additive in both components. This linearity simplifies the structure of $K_2(R, I)$ to a framework compatible with K\"{a}hler differentials, as seen in Theorem \ref{main-thm1}.
\end{remark}

The following lemma is established by iterated applications of (DS1) and (DS3).
\begin{lemma} \cite[p.~255]{alperin1987sk} \label{scholium}
Let $\alpha_0, \dots, \alpha_l$ be elements of a ring such that $1 - \alpha_0 \cdots \alpha_l$ is invertible. Let $\hat{\alpha}_i = \alpha_0 \cdots \alpha_{i-1} \alpha_{i+1} \cdots \alpha_l$. Then
 $$1 = \langle 1, \alpha_0 \cdots \alpha_l \rangle = \prod_{i=0}^l \langle \alpha_i, \hat{\alpha}_i \rangle.$$
\end{lemma}

Given a ring homomorphism $\phi: R \to S$ that endows $S$ with an $R$-algebra structure, the module of K{\"a}hler differentials of $S$ over $R$ which we denoted by $\Omega_{S/R}$, is the $S$-module generated by the set $\{\mathrm{d}(f)|f\in S\}$ subjecting to the following relations:
	\begin{enumerate}[(i)]
		\item $\mathrm{d}(s_1s_2)=s_1\mathrm{d}(s_2)+s_2\mathrm{d}(s_1)$;
		\item $\mathrm{d}(r_1s_1+r_2s_2)=s_1\mathrm{d}(s_2)+s_2\mathrm{d}(s_1)$;
		\item $1\cdot \mathrm{d}(r)=0$ for all $r\in R$.
	\end{enumerate} 
We will abbreviate $\Omega_{S/\mathbb{Z}}$  as $\Omega_{S}$ if there is no confusion.
The following exact sequence is known as the conormal sequence.
\begin{lemma}\cite[Proposition 16.12]{eisenbud2013commutative}\label{Conormal}
	If $\pi: S\rightarrow T$ is an epimorphism of $R$-algebras, with kernel $I$, then there is an exact sequence of $T$-modules
	\begin{equation*}
		I/I^2 \xrightarrow{\mathrm{d}} T\otimes_{S/R}\Omega_{S/R} \xrightarrow{\mathrm{D}\pi} \Omega_{T/R} \rightarrow 0, 
	\end{equation*}
where the right-hand map is given by $\mathrm{D}\pi : c\otimes \mathrm{d}b \mapsto c\mathrm{d}b$ and the left-hand map takes the class of $f$ to $1\otimes \mathrm{d}f$.
\end{lemma}

\sec{2 Some results on K\"{a}hler differentials}
\stepcounter{section}
In this section, we provide two technical results necessary for the following sections. Lemma \ref{Omega} is a special case of \cite[Lemma 1]{magurn2007explicit}, for which we offer an alternative proof. This result will be applied to establish Lemma \ref{K2FpG}. Additionally, Theorem \ref{main-thm1} is required to prove Proposition \ref{excision}.

\begin{lemma}\label{Omega}
	Let $G=\langle g_1\rangle \times \cdots \times \langle g_r\rangle$ be a finite abelian $p$-group.  Then $\Omega_{\mathbb{F}_p[G]}$ is a free $\mathbb{F}_p[G]$-module with basis $\mathrm{d}g_1,\cdots, \mathrm{d}g_r$.
\end{lemma}

\begin{proof}
Let $p^{n_i}$ be the order of $g_i, 1\leq i\leq r$.  Obviously, the $\mathbb{F}_p$-algebra 
$S=\mathbb{F}_p[x_1, \cdots, x_r]$ maps onto $T=\mathbb{F}_p[G]$, and the kernel $I$ is generated by $x_1^{p^{n_1}}-1,\cdots, x_r^{p^{n_r}}-1$. Since $\mathrm{d}(I)=0$,
by Lemma \ref{Conormal}, we have
$$\Omega_{\mathbb{F}_p[G]}\cong \mathbb{F}_p[G]\otimes_{\mathbb{F}_p[x_1, \cdots, x_r]}\Omega_{\mathbb{F}_p[x_1, \cdots, x_r]},$$
then our lemma follows from the fact that $\Omega_{\mathbb{F}_p[x_1, \cdots, x_r]}$ is 
a free $\mathbb{F}_p[x_1, \cdots, x_r]$-module with basis $\mathrm{d}x_1,\cdots,\mathrm{d}x_r$.
\end{proof}

\begin{theorem} \label{main-thm1}
	Let $A$ be a commutative ring, $I$ an ideal of $A$.  Suppose  $J=(b_1\cdots b_r)A+I$ is an ideal of $A$ with $r>1$ such that $b_iJ\subset I$ for each $i$, then
	\begin{equation*}
		K_2(A/I, J/I)\cong J/I\otimes_{A/I}\Omega_{A/I}.
	\end{equation*}	
\end{theorem}

\begin{proof}
Let $p^{n_i}$ be the order of $g_i$ for $1\leq i\leq r$. Consider the surjective $\mathbb{F}_p$-algebra homomorphism from $S=\mathbb{F}_p[x_1, \cdots, x_r]$ to $T=\mathbb{F}_p[G]$. The kernel $I$ is generated by $x_1^{p^{n_1}}-1, \cdots, x_r^{p^{n_r}}-1$. Since $\mathrm{d}(I)=0$, by Lemma \ref{Conormal}, we have
$$
\Omega_{\mathbb{F}_p[G]}\cong \mathbb{F}_p[G]\otimes_{\mathbb{F}_p[x_1, \cdots, x_r]}\Omega_{\mathbb{F}_p[x_1, \cdots, x_r]}.
$$
The result follows from the fact that $\Omega_{\mathbb{F}_p[x_1, \cdots, x_r]}$ is a free $\mathbb{F}_p[x_1, \cdots, x_r]$-module with basis $\mathrm{d}x_1, \cdots, \mathrm{d}x_r$.
\end{proof}

\begin{theorem} \label{main-thm1}
Let $A$ be a commutative ring and $I$ be an ideal of $A$. Suppose $J=(b_1\cdots b_r)A+I$ is an ideal of $A$ with $r>1$ such that $b_iJ\subset I$ for each $i$. Then
\begin{equation*}
    K_2(A/I, J/I)\cong J/I\otimes_{A/I}\Omega_{A/I}.
\end{equation*}    
\end{theorem}

\begin{proof}
Since $b_iJ\subset I$ for each $i$, it follows that $J^2\subset I$; thus, for any $(J/I)^2 = 0$. According to Theorem 1.3 in \cite{weibel1980k2}, there exists an exact sequence of $A/J$-modules:
$$
J/I\otimes_{A/J}J/I\xrightarrow{\psi} K_2(A/I,J/I)\xrightarrow{\rho} J/I\otimes_{A/J}\Omega_{A/J}\rightarrow 0,
$$
where $K_2(A/I,J/I)$ is generated by the Dennis-Stein symbols $\langle a, b\rangle$ with $a$ or $b$ in $J/I$. The map $\psi$ is defined by $a\otimes b \mapsto \langle a, b\rangle$, and by Remark \ref{remak1} it is straightforward to verify that $\psi$ is a bilinear map. The map $\rho$ is given by
$$
\rho \langle a, b\rangle =
\begin{cases}
   a \otimes \mathrm{d}b, & \text{if } a \in J/I; \\
   -b \otimes \mathrm{d}a, & \text{if } b \in J/I.        
\end{cases}
$$

For $s, t \in A$, let $\alpha_0 = s$ and $\alpha_i = b_i$ for $1 \leq i \leq r$. We have
\begin{align*}
    & \quad \psi(b_1b_2\cdots b_{r}s \otimes b_1b_2\cdots b_{r}t)\\
    &= \langle b_1b_2\cdots b_{r}s, b_1b_2\cdots b_rt \rangle \\
    &= \langle s, (b_1b_2\cdots b_r)^2t \rangle \prod_{i=1}^r \langle b_i, (b_1b_2\cdots b_r)^2st/b_i \rangle \tag{by Lemma \ref{scholium}} \\
    &= 1. \tag{since $b_iJ \subset I$}
\end{align*}
The second-last equality follows by regarding the first entry of $\langle b_1 \cdots b_r s, b_1 \cdots b_r t \rangle$ as a product of $r+1$ factors and leaving the second entry unchanged.
The last equality holds because $r > 1$ and $b_iJ \subset I$ for each $i$, so that the elements $(b_1\cdots b_r)^2t$ and $(b_1\cdots b_r)^2st/b_i$ all lie in $I$.
In the relative group $K_2(A/I, J/I)$, these elements represent $0$ in $A/I$, which implies that every Dennis-Stein symbol in the product evaluates to the identity element $1$.

Thus, $\operatorname{Im}(\psi) = 1$, it follows that  $K_2(A/I, J/I)\cong J/I\otimes_{A/J}\Omega_{A/J}$. To complete the proof, it suffices to show that
$J/I\otimes_{A/J}\Omega_{A/J}\cong J/I\otimes_{A/I}\Omega_{A/I}$. According to Lemma \ref{Conormal}, there is an exact sequence of $A/J$-modules:
$$
J/I\xrightarrow{\delta} A/J\otimes_{A/I}\Omega_{A/I} \rightarrow \Omega_{A/J}\rightarrow 0.
$$        
Tensoring the sequence above with $J/I$, we obtain
$$
J/I\otimes_{A/J}J/I \xrightarrow{\delta^{*}}\Omega_{A/I}\otimes_{A/I}J/I\rightarrow J/I\otimes_{A/J}\Omega_{A/J}\rightarrow 0,
$$
where $\delta^{*}$ is a bilinear map. For $s, t \in A$, we have
\begin{align*}
    \delta^{*}(b_1b_2\cdots b_{r}s\otimes b_1b_2\cdots b_{r}t) 
    &= b_1b_2\cdots b_{r}t \otimes \mathrm{d}(b_1b_2\cdots b_{r}s) \\
    &= (b_1b_2\cdots b_r)^2t \otimes \mathrm{d}s + 
    \sum_{i=1}^{r} (b_1b_2\cdots b_r)^2t/b_i\otimes \mathrm{d}b_i \\
     &= 0 \otimes \mathrm{d}s + 
    \sum_{i=1}^{r} 0\otimes \mathrm{d}b_i \\
    &= 0.
\end{align*}
Hence, $J/I\otimes_{A/J}\Omega_{A/J}\cong J/I\otimes_{A/I}\Omega_{A/I}$. This completes the proof of the theorem.    
\end{proof}

\sec{3 On $K_2$-relations for rings in a Cartesian square}
\stepcounter{section}
In this section, we extend several results from \cite[Section 4]{ChenHong2014} concerning $K_2$-relations for rings lying in the Cartesian square \eqref{cartesian}.

Suppose $G = C_{p^{n_1}} \times \cdots \times C_{p^{n_r}}$ is a finite abelian $p$-group with generators $g_1, \dots, g_r$, and let $\widetilde{G} = \sum_{g\in G}g$ denote the sum of all group elements. Let $\Gamma$ be the maximal $\mathbb{Z}$-order in the rational group algebra $\mathbb{Q}[G]$. We define $J = |G|\Gamma$ and $I = J \cap p\mathbb{Z}[G]$ as ideals of $\mathbb{Z}[G]$. As the structure  of $\Gamma$ is well understood \cite[Lemma 3.3]{ChenHong2014},  we have the relations $J = I + \widetilde{G}\mathbb{Z}$ and $p\widetilde{G} \in I$ hold.

\begin{lemma} \label{K2FpG}
Let $G$ be as above with $|G|>2$, and set $x_i=g_i-1$. Then $K_2(\mathbb{F}_p[G],(\widetilde{G}))$ is an elementary abelian $p$-group of rank $r$ with basis $\{\langle x_i , \prod_{j=1}^{r}x_{j}^{p^{n_j}-1}\rangle \mid 1\leq i\leq r\}.$
\end{lemma}

\begin{proof}
Note that $x_{i}G=0$ for all $1 \leq i\leq r$, and the following identity holds in $\mathbb{F}_p[G]$:
$$\widetilde{G}=\prod_{j=1}^{r}\left(\sum_{i=0}^{p^{n_j}-1}(g_j)^i\right)= \prod_{j=1}^{r}(g_j-1)^{p^{n_j}-1}= \prod_{j=1}^{r}x_{j}^{p^{n_j}-1}.$$
Thus, $\widetilde{G}$ has degree greater than 1 if and only if  $|G|>2$. Since $|G|>2$ holds, according to Theorem \ref{main-thm1}, there exists an isomorphism
$$K_2(\mathbb{F}_p[G],(\widetilde{G})) \cong (\widetilde{G}) \otimes_{\mathbb{F}_p[G]} \Omega_{\mathbb{F}_p[G]}.$$
By Lemma \ref{Omega}, $\Omega_{\mathbb{F}_p[G]}$ is a free $\mathbb{F}_p[G]$-module of rank $r$. Since $p\widetilde{G} = 0$ in $\mathbb{F}_p[G]$, the group $K_2(\mathbb{F}_p[G],(\widetilde{G}))$ has exponent $p$, which implies that it is an elementary abelian $p$-group of rank $r$. Consequently, by Lemma \ref{Omega}, the set $\{(\prod_{j=1}^{r}x_{j}^{p^{n_j}-1}) \otimes \mathrm{d} x_i \mid 1\leq i\leq r\}$ constitutes a basis for $(\widetilde{G})\otimes_{\mathbb{F}_p[G]} \Omega_{\mathbb{F}_p[G]}$. Via the map $\rho$ defined in Theorem \ref{main-thm1}, this basis corresponds to a basis of $K_2(\mathbb{F}_p[G],(\widetilde{G}))$, and the result follows.
\end{proof}

\begin{remark}
$K_2(\mathbb{F}_p[G],(\widetilde{G}))$ is trivial when $G=C_2$, as noted in \cite[Lemma 4.1]{ChenHong2014}.
\end{remark}

\begin{proposition}\label{excision}
Let $G,I,J$ be as above with $|G|>2$. Then 
$$K_2(\mathbb{Z}G/I,(\widetilde{G})) \cong
K_2(\mathbb{F}_p[G],(\widetilde{G})).$$
\end{proposition}

\begin{proof}
Let $A = \mathbb{Z}[G]$, $I_1 = I$, $I_2 = pA$, $J_1 = J$, and $J_2 = (\widetilde{G})A + I_2$. For each $1 \leq j \leq r$, define $b_j = \sum_{i=1}^{p^{n_j}} g_j^i$. Then we have $J_1/I_1 = \widetilde{G}(A/I_1)$ and $J_2/I_2 = \widetilde{G}(A/I_2)$. Since $\widetilde{G} = \prod_{j=1}^{r} b_j$ and $p(J_1/I_1)=0$, it follows from Theorem \ref{main-thm1} that $K_2(A/I_1, J_1/I_1)$ is an elementary abelian $p$-group of rank at most $r$, and that there exists a surjective map from $K_2(A/I_1, J_1/I_1)$ to $K_2(A/I_2, J_2/I_2)$. By Lemma \ref{K2FpG}, the $p$-rank of $K_2(A/I_2, J_2/I_2)$ is exactly $r$. Comparing these ranks, we conclude that $K_2(A/I_1, J_1/I_1)$ must also be of rank $r$, and therefore the aforementioned surjection is an isomorphism.
\end{proof}

\begin{remark}
Since $I_1 = J_1 \cap I_2$, we have $(J_1/I_1) \cap (I_2/I_1) = 0$. Using this fact,
by Theorem 14 in \cite{keune1978relativization}, we obtain an exact sequence:
$$ (J_1/I_1)/(J_1/I_1)^2 \otimes_{A^{e}} (I_2/I_1)/(I_2/I_1)^2 \xrightarrow{\psi} K_2(A/I_1, J_1/I_1) \rightarrow K_2(A/I_1, J_2/I_2) \rightarrow 1, $$
where $A^{e} = A \otimes_{\mathbb{Z}} A^{op}$ and the map $\psi$ is defined by $\psi(\overline{a} \otimes_{A^{e}} \overline{b}) = \langle a, b \rangle$. 
In our case, since $(J_1/I_1)^2 = 0$, Remark \ref{remak1} ensures that the Dennis-Stein symbol $\langle a, b \rangle$ is additive in both components. Given that $p\widetilde{G} \in (J_1/I_1) \cap (I_2/I_1) = 0$ in $A/I_1$, for any $s, t \in A$, we have:
$$\psi(\overline{\widetilde{G}s} \otimes_{A^{e}} \overline{pt}) = \langle \widetilde{G}s, pt \rangle = \langle p\widetilde{G}s, t \rangle = \langle 0, t \rangle = 1.$$
Hence, the image of $\psi$ is trivial. This yields an alternative proof of the isomorphism.
\end{remark}

\begin{theorem}\label{main-thm2} Let $G,I,J$ be as above with $r>1$. Then
\begin{equation*}
 K_2(\mathbb{Z}[G]/I)\cong K_2(\mathbb{Z}[G]/J) \oplus K_2(\mathbb{F}_p[G], (\widetilde{G})).
\end{equation*}
\end{theorem}

\begin{proof}
The Cartesian square \eqref{cartesian} gives rise to a natural commutative diagram with exact rows:
\begin{equation*}
\begin{tikzcd}
& K_2(\mathbb{Z}[G]/I, J/I) \arrow[r, "f_1"] \arrow[d, "g_1"] & K_2(\mathbb{Z}[G]/I) \arrow[r, "f_2"] \arrow[d, "g_2"] & K_2(\mathbb{Z}[G]/J) \arrow[r] \arrow[d, "g_3"] & 1 \\
 & K_2(\mathbb{F}_p[G], (\widetilde{G})) \arrow[r, "i"]        & K_2(\mathbb{F}_p[G]) \arrow[r]                         & K_2(\mathbb{F}_p[G]/(\widetilde{G})) \arrow[r] & 1
\end{tikzcd}
\end{equation*}
where the surjectivity of the maps follows from \cite[Proposition 1.1]{alperin1985sk}. The homomorphism $i$ is injective: each generator of $K_2(\mathbb{F}_p[G], (\widetilde{G}))$ has order $p$, and each corresponding image of $i$ denoted by the same literal, also has order $p$ and lies in a basis for $K_2\left(\mathbb{F}_p[G]\right)$ when $r>1$ (see \cite{zhang2019some}). As a result, the bottom exact sequence splits: the monomorphism $i$ has a section, which we denote by $j$. Since  $g_1$ is an isomorphism by Proposition \ref{excision}, the commutativity of the first square implies that $f_1$ is injective.  Consequently, $g_{1}^{-1} j g_{2}$ is a section of $f_1$, showing that $f_1$ is a split injection. The top exact sequence therefore splits, and we obtain
\begin{equation*}
  K_2(\mathbb{Z}[G]/I)\cong K_2(\mathbb{Z}[G]/J) \oplus K_2(\mathbb{F}_p[G], (\widetilde{G})).
\end{equation*}
\end{proof}
\begin{remark}
By the Snake Lemma, we have  $\mathrm{Ker}(g_2)\cong \mathrm{Ker}(g_3)$. In addition, the corresponding square for $K_2$​-groups in \eqref{cartesian} is Cartesian in the category of groups.
\end{remark}

\sec{Acknowledgement}

We thank the referees for their time and comments.
 \sec{References} \footnotesize \baselineskip
13pt
{

\end{document}
\parindent=0pt
\vspace*{-1.2cm}

\begin{thebibliography}{99}


\bibitem{smallK2}
Zhang, Y. and Tang, G.P., {$K_2$} for a kind of special finite group rings, {\em J. Univ. Chinese Acad. Sci.}, 2018, 35(6):721.

\bibitem{alperin1985sk}
Alperin, R.C., Dennis, R.K. and Stein, M.R., ${SK_1}$ of finite abelian groups. {I}, {\em Invent. Math.}, 1985, 82(1):1--18.

\bibitem{ChenHong2014}
Chen, H., Gao, Y.B. and Tang, G.P., ${K_2}$ of a Quotient Ring of $\mathbb{Z}{G}$, {\em Comm. Algebra}, 2014, 42(4):1571--1581.

\bibitem{weibel2013kbook}
Weibel, C.A., {\em The {$K$}-book: An Introduction to Algebraic {$K$}-theory}, volume 145 of {\em Grad. Stud. Math.}. Amer. Math. Soc., Providence, RI, 2013.

\bibitem{weibel1980k2}
Weibel, C.A., ${K_2}$, ${K_3}$ and nilpotent ideals, {\em J. Pure Appl. Algebra}, 1980, 18(3):333--345.

\bibitem{alperin1987sk}
Alperin, R.C., Dennis, R.K., Oliver, R. and Stein, M.R., ${SK_1}$ of finite abelian groups. {II}, {\em Invent. Math.}, 1987, 87(2):253--302.

\bibitem{eisenbud2013commutative}
Eisenbud, D., {\em Commutative Algebra: With a View Toward Algebraic Geometry}. Springer, 2013.

\bibitem{magurn2007explicit}
Magurn, B.A., Explicit ${K_2}$ of some finite group rings, {\em J. Pure Appl. Algebra}, 2007, 209(3):801--811.


\bibitem{keune1978relativization}
Keune, F., The relativization of ${K_2}$, {\em J. Algebra}, 1978, 54(1):159--177.

\bibitem{zhang2019some}
Zhang, H., Tang, G.P. and Liu, H., Some remarks on ${K_2}$ and ${K_3}$ of finite abelian group algebras, {\em J. Algebra Appl.}, 2019, 18(5):1950094.

\end{thebibliography}
